\documentclass[12pt]{amsart}

\usepackage{hyperref}
\usepackage{fullpage}
\usepackage{amsmath}
\usepackage{amsthm}
\usepackage{mathrsfs}
\usepackage{amssymb}
\usepackage{comment}
\newtheorem{theorem}{Theorem}[section]
\newtheorem{lemma}[theorem]{Lemma}
\newtheorem{proposition}[theorem]{Proposition}
\newtheorem{example}[theorem]{Example}
\newtheorem{corollary}[theorem]{Corollary}

\newtheorem*{conjecture}{Conjecture}
\newtheorem{hyp}[theorem]{Hypothesis}
\theoremstyle{definition}
\newtheorem*{remark}{Remark}
\newtheorem{definition}[theorem]{Definition}

\begin{document}
\title{Bounded gaps between primes in special sequences}
\author{Lynn Chua}
\address{Department of Mathematics, Massachusetts Institute of Technology, 3 Ames St., Cambridge, MA, 02139}
\email{chualynn@mit.edu}
\author{ Soohyun Park}
\address{Department of Mathematics, Massachusetts Institute of Technology, 3 Ames St., Cambridge, MA, 02139}
\email{soopark@mit.edu}
\author {Geoffrey D. Smith}
\address{Department of Mathematics, Yale University, 10 Hillhouse Avenue New Haven, Connecticut, 06511}
\email{geoffrey.smith@yale.edu}
\begin{abstract}

We use Maynard's methods to show that there are bounded gaps between primes in the sequence $\{\lfloor n\alpha\rfloor\}$, where $\alpha$ is an irrational number of finite type. In addition, given a superlinear function $f$ satisfying some properties described by Leitmann, we show that for all $m$ there are infinitely many bounded intervals  containing $m$ primes and at least one integer of the form $\lfloor f(q)\rfloor$ with $q$ a positive integer.
\end{abstract}
\maketitle
\section{Introduction and statement of results} \label{sectintro}

The famous Twin Prime Conjecture states that there are infinitely many pairs of primes of the form $p, p + 2$. A natural generalization of this problem concerns finding $k$-tuples of primes generated by linear forms. We first define a necessary condition for such linear forms.

Let $\mathcal{L}=\{L_1,\ldots,L_k\}$ be a finite set of integer-valued linear functions $L_i(n) = a_i n+b_i$. We say $\mathcal{L}$ is \emph{admissible} if for any prime $p$ there is an integer $n$ such that  $p$ does not divide $\prod_{1\leq i\leq k} L_i(n)$. The analogue for the Twin Prime Conjecture for $k$-element sets of linear forms is as follows.

\begin{conjecture}[Prime $k$-tuples Conjecture] \label{k-tup}
If $\mathcal{L}=\{L_1,\ldots,L_k\}$ is admissible, then there are infinitely many integers $n$ such that $L_i(n)$ is prime for all $1\leq i\leq k$.
\end{conjecture}

There has been recent success in demonstrating weakened forms of the Prime $k$-tuples Conjecture, in which \emph{several} of the $L_i(n)$, rather than all, are required to be prime. These results depended on obtaining tight estimates on the distribution of primes in arithmetic progressions. If we denote the $n^{\text{th}}$ prime number by $p_n$, then Goldston, Pintz, and Y{\i}ld{\i}r{\i}m \cite{GPY09} showed that 
\begin{equation} \liminf\limits_{n\rightarrow \infty} \frac{p_{n+1}-p_n}{\log p_n} = 0\,, \end{equation}
using what we now call the GPY method. Their proof depends heavily on the distribution of primes in arithmetic progressions. In particular, they showed that it is possible to obtain bounded gaps between primes by assuming a stronger version of the Bombieri-Vinogradov Theorem, a result on the distribution of primes in arithmetic progressions. A variant of the Bombieri-Vinogradov Theorem, leading to a refined version of the GPY method, was used by Zhang in \cite{Zha14} to prove the breakthough result that 
\begin{equation}
 \liminf\limits_{n \rightarrow \infty} (p_{n + 1} - p_n) \le 70\text{ }000\text{ }000\,, 
 \end{equation} 
Later, Maynard  \cite{May13} showed that 
\begin{equation} \liminf\limits_{n \rightarrow \infty} (p_{n + 1} - p_n) \le 600\,, \end{equation} 
by using an improvement on the GPY method which involved choosing more general sieve weights. In addition, he shows that 
\begin{equation} \liminf\limits_n (p_{n+m}-p_n)\ll m^3 e^{4m}\,,\end{equation}
a result was which also proven independently by Tao. The Polymath8b project \cite{P8B} has subsequently improved these bounds.

Previous work by Thorner has shown the existence of bounded gaps between primes in Chebotarev sets \cite{Tho14}. Recently, Maynard \cite{May14} proved a generalization of his previous work which showed that any subset of the primes which is ``well-distributed'' in arithmetic progressions contains many primes which are ``close together". The distribution conditions he assumes for his main results are outlined in Hypothesis \ref{Hypothesis1} of Section \ref{sectprelim}. We adapt his methods to show the existence of bounded gaps in certain subsequences of the primes. 

One such set consists of the set of {\em Beatty primes}, which are primes contained in \emph{Beatty sequences} of irrational numbers of \emph{finite type}. Let $\mathbb{Z}^+$ denote the positive integers, and let $\mathbb{P}$ denote the set of primes. Let $\lfloor x \rfloor$ denote the largest integer less than or equal to $x$ and let $\langle x \rangle$ denote the minimum distance from $x$ to an integer. A Beatty sequence is a sequence of the form $\{\lfloor\alpha n \rfloor\}_{n \ge 1}$, where $\alpha$ is an irrational number. The irrational number $\alpha$ is of finite type if 
\begin{equation} \tau=\sup\{r: \liminf_{n\in \mathbb{Z}^+ } n^r \langle n\alpha\rangle=0\} \end{equation}
is finite. For example, all real algebraic numbers and $\pi$ are of finite type. 

We have the following result on Beatty sequences.
\begin{theorem} \label{BeattyBound}
Let $\alpha$ be an irrational number of finite type, and fix $m\in\mathbb{Z}^+$.  There exist constants $\Delta_{\alpha,m} \in\mathbb{Z}^+$ and $B > 0$ depending only on $\alpha$ and $m$ such that the bound
$$
 \#\{ x\leq n< 2x:\, \mbox{there exist m distinct primes of the form $\lfloor \alpha r \rfloor$}, r\in[n, n + \Delta_{\alpha,m}]\} \gg \frac{x}{(\log x)^B}$$
 holds.
\end{theorem}

The following corollary is then immediate.

\begin{corollary}\label{ShortBeatty}
Let $\alpha>1$ be a fixed irrational number of finite type. For any $m\in \mathbb{Z}^+$, there exists a constant $\Delta_{\alpha,m} \in \mathbb{Z}^+$ such that for infinitely many $n$ there are $m$ distinct primes of the form $\lfloor \alpha q\rfloor$ with $n\leq q \leq n+\Delta_{\alpha,m}$.
\end{corollary}

%
%
We prove a similar result for a broad family of sequences $\{\lfloor f(n)\rfloor\}$, for well-behaved differentiable functions $f$ studied by Leitmann \cite{Lei77}. We will call these functions \emph{Leitmann} functions, and we will call primes $p=\lfloor f(n)\rfloor$ {\em Leitmann primes}. One property of these functions is that they grow more quickly than $n$ (see Section 2.2). Some examples of Leitmann functions include $f(n)=n^\Gamma$ for some $1<\Gamma<12/11$,  $f(n) = n \log \log n$ and $f(n) = n \log^C n$ for some $C > 0$ \cite{Lei77}. 


\begin{theorem}\label{LeitmannBound}
Fix $m\in\mathbb{Z}^+$, let $f(n)$ be a Leitmann function and set $g(y):=f^{-1}(y)$. Let $\mathcal{L} = \{L_1,\ldots,L_k\}$ be an admissible set of linear forms $L_i(n)=n+l_i$. There exists a positive constant $C$  depending only on $f$ such that if $k>e^{Cm}$ we have
\begin{eqnarray*} \#\{ x\leq f(n) <2x : \, \#(\{L_1(f(n)), \ldots, L_k(f(n))\} \cap \mathbb{P}) \ge m \} 
\gg \frac{g(2x)-g(x)}{(\log x)^k\exp(Ck)} \,,\end{eqnarray*}
where the implied constant depends only on $f$.
\end{theorem}

\begin{corollary}\label{ShortLeitmann}
Let $f(n)$ be a Leitmann function. For any positive integer $m$, there exists a $\Delta_{f,m} \in \mathbb{Z}^+$ such that there 
exist infinitely many $n$ such that the interval $[n,n+\Delta_{f,m}]$ contains $m$ primes and an integer of the form $\lfloor f(q)\rfloor$ for some $q\in \mathbb{Z}^+$.
\end{corollary}

\begin{remark}
In Theorem \ref{LeitmannBound}, the primes are not necessarily of the form $\lfloor f(m)\rfloor$ for $m\in\mathbb{N}$, in contrast to Theorem \ref{BeattyBound}. The weaker result in this case is a consequence of the fact that $\lim\limits_{n\rightarrow \infty} \frac{f(n)}{n} = \infty$, as compared to Beatty sequences.
\end{remark}
\begin{remark}
The constants $\Delta_{\alpha,m}$, $\Delta_{f,m}$, $B$, and $C$ above are effectively computable given a choice of $f$ or $\alpha$, but we do not evaluate them explicitly in this paper.
\end{remark}

This paper is organized as follows. In Section \ref{sectprelim}, we give a more detailed overview of Maynard's methods, followed by an overview of Leitmann functions and results on uniform distribution of sequences. In Section \ref{sectbeatty}, we prove results on Beatty sequences in order to prove Theorem \ref{BeattyBound}. In Section \ref{sectleitmann}, we use analogous considerations to prove Theorem \ref{LeitmannBound}.

\section{Preliminaries} \label{sectprelim}
We now give an overview of Maynard's methods \cite{May14} and how it applies to our work, followed by definitions and properties of Leitmann functions and related results. In what follows, we use $\{x\} := x - \lfloor x\rfloor$ to denote the {\em fractional part} of $x$.

\subsection{Overview of Maynard's methodology}
We first define a term relating to the distribution of primes in arithmetic progressions. Let $\pi(x; q, a)$ be the number of primes $\le x$ which are congruent to $a$ (mod $q$).

\begin{definition}The primes are said to have a {\em level of distribution} $\theta$ for some $\theta > 0$ if
\begin{equation*}
\sum_{q \le x^\theta} \max_{(a, q) = 1} \left|\pi(x; q, a) - \frac{\pi(x)}{\phi(q)}\right| \ll_{A} \frac{x}{(\log x)^A}\,
\end{equation*} 
for every $A > 0$. \end{definition}
The following theorem is a deep and celebrated theorem in analytic number theory.
\begin{theorem}[Bombieri-Vinogradov Theorem, \cite{Dav}]
The primes have a level of distribution $\theta$ for any $0<\theta <\frac{1}{2}$. \end{theorem} 

In Maynard's work on more general subsets of primes, he shows that if we assume a technical hypothesis, which includes a Bombieri-Vinogradov-type result together with other assumptions of uniform distribution in residue classes, then we get a result on bounded gaps between primes. Specifically, let $\mathcal{A}$ be a set of positive integers. Let $\mathcal{P}$ be a set of primes, and let $\mathcal{L}=\{L_1,\ldots,L_k\}$ be a finite admissible set of linear forms of the form $L_i(n)=n+l_i$, $1\leq i\leq k$. We define
\begin{eqnarray}
\mathcal{A}(x) &:=& \{n \in \mathcal{A} : x \le n < 2x \}\,, \\
L(\mathcal{A}) &:=& \{L(n): n \in \mathcal{A}\}  \,,\\
\mathcal{P}_{\mathcal{L}, \mathcal{A}}(x) &:=& L(\mathcal{A}(x)) \cap \mathcal{P}\,, \\
\mathcal{A}(x; q, a) &:=& \{n \in \mathcal{A}(x) : n \equiv a \text{ (mod $q$)}\}\,,\\
\mathcal{P}_{\mathcal{L}, \mathcal{A}}(x; q, a) &:=& L(\mathcal{A}(x; q, a)) \cap \mathcal{P}\,.
\end{eqnarray}

Maynard then considers sets $\mathcal{A},\mathcal{P},\mathcal{L}$ satisfying the following hypothesis for some $\theta>0$.
\begin{hyp}[\cite{May14}, Hypothesis 1]\label{Hypothesis1}
The following conditions hold.
\begin{enumerate}
\item \label{hyp1} $\mathcal{A}$ is well-distributed in arithmetic progressions. That is,
$$ \sum_{q\leq x^\theta} \max_a \left|\#\mathcal{A}(x;q,a) - \frac{\#\mathcal{A}(x)}{q}\right| \ll \frac{\#\mathcal{A}(x)}{(\log x)^{100k^2}}\,. $$

\item \label{hyp2} Primes in $L(\mathcal{A})\cap \mathcal{P}$ are well-distributed in most arithmetic progressions. For any $L\in\mathcal{L}$, we have
$$ \sum_{\substack{q\leq x^\theta }} \max_{(L(a),q) = 1} \left| \#\mathcal{P}_{L,\mathcal{A}}(x;q,a) - \frac{\#\mathcal{P}_{L,\mathcal{A}}(x)}{\phi(q)}\right| \ll \frac{\#\mathcal{P}_{L,\mathcal{A}}(x)}{(\log x)^{100k^2}}\,.$$

\item \label{hyp3} $\mathcal{A}$ is not too concentrated in any arithmetic progression. Namely, for any $q<x^\theta$, we have
$$\mathcal{A}(x;q,a) \ll \frac{\#\mathcal{A}(x)}{q}\,.$$
\end{enumerate}
\end{hyp}

If Hypothesis \ref{Hypothesis1} holds, then Maynard \cite{May14} proves the following.

\begin{theorem}[\cite{May14}, Theorem 3.1]\label{thmmaynard} Let $\alpha>0$ and $0<\theta<1$. There is a constant $C$ depending only on $\theta$ and $\alpha$ such that the following holds.
Let $(\mathcal{A},\mathcal{L},\mathcal{P}, 
\theta)$ satisfy Hypothesis \ref{Hypothesis1}, for some set of integers $\mathcal{A}$, set of primes $\mathcal{P}$, and admissible set of linear functions $\mathcal{L}$. 
Assume $k=\#\mathcal{L}$ satisfies $C\leq k\leq (\log x)^\alpha$ and the coefficients $a_i,b_i$ of $L_i$ satisfy $a_i,b_i\leq x^\alpha$ for all $1\leq i\leq k$.

If $\delta > (\log k)^{-1}$ is such that 
\begin{equation} \frac{1}{k}
\sum_{L\in\mathcal{L}} \frac{\phi(a_i)}{a_i} \#\mathcal{P}_{L,\mathcal{A}}(x) \geq \delta \frac{\#\mathcal{A}(x)}{\log x}\,, \end{equation}
then 
\begin{equation}
 \#\{ n\in\mathcal{A}(x)\,:\, \#(\{L_1(n),\ldots,L_k(n)\}\cap\mathcal{P})\geq C^{-1}\delta \log k\} \gg \frac{\#\mathcal{A}(x)}{(\log x)^k \exp(Ck)}\,. 
 \end{equation}
Moreover, if $\mathcal{P}=\mathbb{P}$, $k\leq (\log x)^{1/5}$ and all $L\in\mathcal{L}$ have the form $an+b_i$ with $|b_i| \leq (\log x)k^{-2}$ and $a\ll 1$, then the primes counted above can be restricted to be consecutive, at the cost of replacing $\exp(Ck)$ with $\exp(Ck^5)$ in the bound.
\end{theorem}

It will follow from Theorem \ref{thmmaynard} that if Hypothesis \ref{Hypothesis1} holds for some positive constant $\theta$ for our particular choices of $\mathcal{A}$, $\mathcal{P}$ and $\mathcal{L}$, then we have a result on bounded gaps between tuples of primes as desired.

\subsection{Leitmann functions}
In this paper, we show results on bounded gaps in sparse subsets of primes. We consider the class of Leitmann functions which grow more quickly than $n$. In Leitmann's work \cite{Lei77}, he proves Bombieri-Vinogradov type results for the following class of functions.

\begin{definition} \label{defleitmanntype} Let $\mathcal{F}$ be the set of {\em Leitmann} functions $f:[c_0,\infty) \rightarrow [2,\infty)$, with $c_0\geq 1$, which have continuous derivatives up to third order such that $f'(x)>0$, $f''(x)>0$ for $x\geq c_0$ and 
\begin{equation}
xf^{(i)}(x) = f^{(i-1)}(x)(\alpha_i+o(1))\,, \end{equation}
for $i=1,2,3$, where $\alpha_1>0$, $\alpha_2\geq 0$, $\alpha_1\neq\alpha_2$, $\alpha_3\neq 3\alpha_1$, $2\alpha_1-3\alpha_2 + \alpha_3\neq 0$. If $\alpha_2 = 0$, then for $x\geq c_0$,
\begin{equation} xf''(x) = f'(x) s(x) t(x)\,, \end{equation}
where $s(x)$ is positive, non-increasing, and tends to zero if $x$ tends to infinity. Moreover, $s^{-1}(x) \ll_\epsilon x^\epsilon$ for each $\epsilon>0$ and $0<c_1\leq t(x) \leq c_2$ for some constants $c_1,c_2$. $f$ also satisfies $x=o(f(x))$ and $f(x)\ll x^{12/11-\epsilon_1}$ for some positive $\epsilon_1$. 
\end{definition}

\begin{example} The following functions are Leitmann:
\begin{equation*}
x^\beta (\log x)^A\,, x^\beta \exp(A(\log x)^\beta)\,, x(\log x)^C\,, x\exp(C(\log x)^B)\,, xl_m(x)\,,
\end{equation*}
where $A$ is real, $1<\beta <12/11$, $0<B<1$, $C>0$, $l_1(x) = \log x$ and $l_{m+1}(x) = \log(l_m(x))$. \end{example}

Leitmann proved the following Bombieri-Vinogradov-type result relating to the distribution of prime numbers of the form $p=\lfloor f(n)\rfloor$ for $n\in\mathbb{N}$. 
\begin{theorem}[\cite{Lei77}, Theorem 1.1] \label{leitmannthm}
Let $f:[c_0,\infty)\rightarrow \mathbb{R}^+$ be Leitmann, let $\pi_f(x;q,a)$ be the number of primes $p\leq x$ such that $p=\lfloor f(n) \rfloor$ for some $n\in\mathbb{N}$, and $p\equiv a$ (mod $q$). Let $x\geq c = f(c_0)$, let $\phi$ be Euler's totient function, and let $g$ be the inverse function of $f$. The following are true.
\begin{enumerate}
\item \label{thmleitmannbound1} If $A>0$, $1\leq q \leq (\log x)^A$, $(a,q) = 1$ and $b$ is some positive constant, then
\begin{equation*} \pi_f(x; q, a) := \sum_{\substack{p\leq x,\,n\in\mathbb{N}\\ p=\lfloor f(n)\rfloor \\ p\equiv a\, (\text{mod } q)}} 1 
= \frac{1}{\phi(q)} \int_c^x \frac{g'(t)}{\log t}\, dt + O\left(g(x) \exp\left\{ -b\sqrt{\log x}\right\}\right)\,. \end{equation*}

\item \label{Leitmann-Vinogradov} For every $A>0$, there exists some $\theta>0$ such that
\begin{equation*}\sum_{q\leq x^{\theta}} \max_{(a,q)=1} \max_{c\leq y\leq x} \left| \pi_f(y;q,a) - \frac{1}{\phi(q)} \int_c^y \frac{g'(t)}{\log t}\, dt\right| \ll \frac{g(x)}{(\log x)^A}\,. \end{equation*}
\end{enumerate}
\end{theorem}

\subsection{Uniform distribution of sequences modulo 1}
Part \ref{Leitmann-Vinogradov} of Theorem \ref{leitmannthm} is equivalent to condition \ref{hyp2} of Hypothesis \ref{Hypothesis1}, hence to prove Hypothesis \ref{Hypothesis1}, it suffices to prove conditions \ref{hyp1} and \ref{hyp3}. To do this, we use results on uniform distribution of sequences. We first define the discrepancy of a sequence.

\begin{definition}[\cite{KN12}] Let $x_1,\ldots, x_N$ be a finite sequence of real numbers, and for any $E\subseteq [0,1)$  let $A(E;N)$ denote the number of terms $x_n$, $1\leq n \leq N$, for which $\{x_n\}\in E$. The {\em discrepancy} of the sequence is the number 
\begin{equation*}
D_N(x_1,\ldots, x_N) = \sup_{0\leq \alpha <\beta \leq 1} \left| \frac{A([\alpha,\beta);N)}{N} - (\beta-\alpha)\right| \,, \end{equation*}
For an infinite sequence $\omega$ of real numbers, or a finite sequence containing at least $N$ terms, the discrepancy $D_N(\omega)$ is the discrepancy of the initial segment formed by the first $N$ terms of $\omega$.
\end{definition}

We can bound the discrepancy of a sequence using a theorem of Erd\H{o}s and Tur\'an.

\begin{theorem} [\cite{KN12}, Chapter 2, Theorem 2.5]\label{ETT} There is an absolute constant $C>0$ such that for any finite sequence $x_1,\ldots,x_N$ of real numbers and any positive integer $m$ we have 
\begin{equation*} 
D_N \leq C\left(\frac{1}{m} + \sum_{h=1}^m \frac{1}{h}\left| \frac{1}{N} \sum_{n=1}^N e^{2\pi i h x_n} \right| \right)\,. \end{equation*}
\end{theorem}

We also use the following theorem for bounding exponential sums.

\begin{theorem} [\cite{KN12}, Chapter 1, Theorem 2.7]\label{2.7}
 Let $a$ and $b$ be integers with $a<b$, and let $f$ be twice differentiable on $[a,b]$ with $f''(x)\geq \rho >0$ or $f''(x)\leq -\rho <0$ for $x\in[a,b]$. Then
\begin{equation*}
\left| \sum_{n=a}^b e^{2\pi i f(n)}\right| \leq \left(\left|f'(b) - f'(a)\right| + 2\right) \left(\frac{4}{\sqrt{\rho}} + 3\right) \,. \end{equation*} 
\end{theorem}
We will eventually use Theorems \ref{ETT} and \ref{2.7} to prove  Parts \ref{hyp1} and \ref{hyp3} of Hypothesis \ref{Hypothesis1} for for integers of the form $\lfloor f(n)\rfloor$ with $f$ a Leitmann function.

\section{Proof of Theorem \ref{BeattyBound}} \label{sectbeatty}

Theorem \ref{BeattyBound} will follow from the following key result, whose proof we defer to the second part of the section.
\begin{theorem}\label{BeattyHyp}
Let $\alpha$ be an irrational number of finite type, fix some $c\in (0,1]$, and let $\mathcal{L}=\{L_1,\ldots, L_k\}$ be a set of linear forms $L_i$ each of the form $L_i(n)=n+l_i$. Set $\mathcal{A}:=\{\lfloor \alpha n\rfloor: \{\alpha n\} <c\}$, and let $\mathcal{P}$ be the set of primes. Then $(\mathcal{A},\mathcal{P},\mathcal{L}, \theta)$ satisfies Hypothesis \ref{Hypothesis1} for some $\theta>0$ independent of the choice of $\mathcal{L}$.
 \end{theorem}
 Assuming Theorem \ref{BeattyHyp}, we now establish Theorem \ref{BeattyBound}.
 \begin{proof}[Deduction of Theorem \ref{BeattyBound} from Theorem \ref{BeattyHyp}]
 We start with a classical result in Diophantine approximation \cite{KN12}. Let $\alpha$ be an irrational number of finite type $\tau$. Then, for every $\epsilon>0$, the discrepancy $D_N(\alpha):=D_N(\alpha,2\alpha,\ldots,N\alpha)$ satisfies
\begin{equation} \label{BeattyUniformDistribution}
 D_N(\alpha)\ll N^{(-1/\tau)+\epsilon}\,. \end{equation}

 Let $\mathcal{A}:=\{\lfloor \alpha n\rfloor: \{\alpha n\} <1/2\}$. By (\ref{BeattyUniformDistribution}), we have $\lvert \mathcal{A} \cap \{1,\ldots,N\} \rvert \sim \frac{N}{2\alpha}$; in particular, $\mathcal{A}$ is infinite. 
 
 We now construct a suitable admissible set $\mathcal{L}=\{L_1,\ldots,L_k\}$ of any given size $k\in \mathbb{Z}^+$. Let $W$ be the product of all primes $p\leq k$. Then, there is at least one residue class $a \pmod W$ such that infinitely many $n\in \mathcal{A}$ satisfy $n\equiv a\pmod W$. Pick $k$ distinct elements $l_1,\ldots,l_k\in \mathcal{A}$ such that $l_i\equiv a\pmod{q}$. Then the set $\mathcal{L}=\{L_1,\ldots, L_k\}$ with $L_i(n)=n+l_i$ is admissible, as for any $n$ the $L_i(n)$ represent only one residue class modulo $p$ for each $p\leq k$ and at most $k$ such classes modulo $p>k$. 

Our choice of $\mathcal{L}$ has the important property that if $n\in \mathcal{A}$, then $L_i(n)$ is in the Beatty sequence $\{\lfloor \alpha m\rfloor : m\in \mathbb{Z}^+\}$. For we have $n=\lfloor \alpha m\rfloor$ and $l_i=\lfloor \alpha m_i\rfloor $ for some $m, m_i\in \mathbb{Z}^+$  with $\{\alpha m\}<1/2$ and $\{\alpha m_i\}<1/2$, so we have $\alpha(m+m_i)=L_i(n)+\{\alpha m\}+\{\alpha m_i\}<L_i(n)+1$, whence $L_i(n)=\lfloor \alpha (m+m_i)\rfloor$ is in the Beatty sequence. 

We now apply Theorem \ref{thmmaynard}. We have by Theorem \ref{BeattyHyp} that the set $(\mathcal{A}, \mathcal{P}, \mathcal{L}, \theta)$ satisfies Hypothesis \ref{Hypothesis1} for some $\theta>0$. Moreover, from the Prime Number Theorem for primes in Beatty sequences of \cite{BS09}, we have the asymptotic bound 
$$
\#\mathcal{P}_{L_i,\mathcal{A}}(x)\sim \frac{x}{\alpha \log x} \sim \frac{\#\mathcal{A}(x)}{\log x}\,.
$$
As such, choosing $\epsilon>0$ and setting $\delta=1-\epsilon$,  for sufficiently large $x$ we have the inequality
$$
\frac{1}{k} \sum_{L\in \mathcal{L}} \#\mathcal{P}_{L_i,\mathcal{A}}(x) \geq \delta \frac{\#\mathcal{A}(x)}{\log x}\,.
$$
Then Theorem \ref{thmmaynard} implies that there exists some constant $C$ independent of $\mathcal{L}$ such that the bound
$$
\#\{n\in \mathcal{A}(x): \#(\{L_1(n),\ldots,L_k(n)\} \cap \mathcal{P}) \geq C^{-1}\delta \log k\} \gg \frac{\# \mathcal{A}(x)}{(\log x)^k e^{Ck}}
$$
holds. Fix any $m\in \mathbb{Z}^+$. We now set $k$  such that $m\leq C^{-1} \delta \log k$, and set of linear forms $\mathcal{L}=\{L_1,\ldots,L_k\}$ as described above. Then the above bound counts the number of $n\in \mathcal{A}$ such that at least $m$ of the integers $L_1(n),\ldots, L_k(n)$ are prime. However, if we have $n\in \mathcal{A}$, we know that $L_i(n)$ is in the Beatty sequence $\{\lfloor \alpha m\rfloor :m\in \mathbb{Z}^+\}$ for all $i$, so setting $\Delta_{\alpha,m}=\max_{1\leq i,j\leq k} \lvert l_i-l_j\rvert$ and $B=k$ we obtain the desired result.
\end{proof}
\begin{remark}
Using the second part of Theorem \ref{thmmaynard}, the above argument also can show that there are infinitely many bounded sets of \emph{consecutive} primes in a given Beatty sequence.
\end{remark}

It now remains to prove Theorem \ref{BeattyHyp}. We start by recalling a critical lemma from \cite{BS09}.

\begin{lemma}[\cite{BS09}, Lemma 4.2]\label{BSlemma}
Let $\alpha$ be a fixed irrational number of finite type $\tau$.  Then for every real number $0<\epsilon<1/(8\tau)$, there is a positive number $\eta$ such that for all integers $1\leq k\leq M^{\epsilon}$ and $0\leq a<q \leq M^{\epsilon/4}$ we have the bound
$$
\Bigl\lvert \sum_{\substack{m\leq M\\m\equiv a\pmod{q}}} \Lambda(m)e(\gamma k m)\Bigr\rvert \leq M^{1-\eta}\,,
$$
if $M$ is sufficiently large.
\end{lemma}
We use this lemma to prove an analogue of the Prime Number Theorem for primes in $\mathcal{A}$.

\begin{lemma}\label{PNTBeatty}
Let $\alpha$ and $\beta$  be fixed real numbers with $\alpha>1$ irrational, and of finite type, and fix $c\in(0,1]$ There exists a constant $\kappa>0$ depending only on $\alpha$ such that for all integers $0\leq a<q\leq N^\kappa$ with $(a,q)=1$, we have 
$$
\sum_{\substack{n\leq N\\\lfloor \alpha n+\beta \rfloor \equiv a \pmod{q}\\ \{\alpha n+\beta \} <c}}\Lambda \left(\lfloor \alpha n+\beta\rfloor\right)
= c\alpha^{-1}\sum_{\substack{m\leq \lfloor\alpha N+\beta\rfloor\\m\equiv a\pmod{q}}} \Lambda(m)+O(N^{1-\kappa})\,,
$$
where the implied constant depends only on $\alpha$, $\beta$, and $c$ and $\Lambda$ denotes the von Mangoldt function.
\end{lemma}

\begin{proof}
This proof follows the same method as the proof of Theorem 5.1 given in \cite{BS09}. We first rewrite the sum we wish to evaluate in a form more amenable to analysis, using the fact that $m$ can be written in the form $\lfloor \alpha n+\beta\rfloor$ with $\{\alpha n+\beta\} <c$ if and only if $0<\{\frac{m-\beta+c}{\alpha}\}\leq c\alpha^{-1}$.  Let $M=\lfloor \alpha N+\beta\rfloor$. We have
\begin{align*}
S_{\alpha,\beta,c;q,a}(N)
&:=\sum_{\substack{n\leq N\\\lfloor \alpha n+\beta \rfloor \equiv a \pmod{q}\\ \{\alpha n+\beta \} < c}}\Lambda \left(\lfloor \alpha n+\beta\rfloor\right)\\
&= \sum_{\substack{\lfloor\beta\rfloor \leq m\leq M \\0< \left\{ \frac{m-\beta+c}{\alpha}\right\}\leq c\alpha^{-1}\\m\equiv a\pmod{q}}}\Lambda (m)\\
&= \sum_{\substack{ m\leq M \\0<\left\{ \frac{m-\beta+c}{\alpha}\right\}\leq c\alpha^{-1}\\ m\equiv a \pmod{q}}}\Lambda (m)+O(1)\,.
\end{align*}
Let $\gamma=c\alpha^{-1}$, let $\delta=\alpha^{-1}(c-\beta)$, and let $\psi(x)$ be the periodic function with period 1 defined on $[0,1)$ by
$$
\psi(x)=\left\{
	\begin{array}{ll}
		1  & \mbox{if } 0< x \leq \gamma\\
		0 & \mbox{if }  \gamma< x\leq1.
	\end{array}
\right.
$$
At once we have 
\begin{equation}\label{FirstBeattyBound}
S_{\alpha,\beta,c;q,a}(N)=\sum_{\substack{m\leq M\\m\equiv a \pmod{q}}} \Lambda(m)\psi(\alpha^{-1}m+\delta) +O(1).
\end{equation}

Let $e(x)=e^{2\pi i x}$. We wish to approximate $\psi(x)$ in terms of a Fourier series $\sum_{k\in \mathbb{Z}} a_k e(kx)$. To do so, we use a result of Vinogradov \cite[Chapter I, Lemma 12]{Vin04} which states that for any $\Delta$ such that $0< \Delta <1/8$ and $\Delta\leq \min(\gamma,1-\gamma)/2$ we have a real-valued function $\psi_\Delta(x)$ that satisfies the following properties:
\begin{enumerate}
\item $\psi_{\Delta}(x)$ is periodic with period one;\\
\item $0\leq \psi_\Delta(x)\leq 1$ for all $x$;\\
\item $\psi_\Delta(x)=\psi(x)$ if $\Delta \leq x \leq \gamma-\Delta$ or if $\gamma+\Delta\leq x\leq 1-\Delta$;\\
\item $\psi_\Delta(x)$ has a Fourier series
$$
\psi_\Delta(x)=\gamma+\sum_{k=1}^{\infty} (g_ke(kx)+h_ke(-kx))
$$
where the coefficients $g_k, h_k$ satisfy the bound
\begin{equation}\label{gkhk}
\max(\lvert g_k\rvert, \lvert h_k \rvert )\leq \min\left(\frac{2}{\pi k}, \frac{2}{\pi^2k^{2}\Delta}\right)\,.
\end{equation}
\end{enumerate}
As such, we can replace $\psi$ with $\psi_\Delta$ in (\ref{FirstBeattyBound}), at a cost of admitting some additional error. Let $\mathcal{I}=[0,\Delta)\cup (\gamma-\Delta,\gamma+\Delta)\cup (1-\Delta,1)$ and let $V(\mathcal{I},M)$ denote the number of $m\leq M$ such that $m\equiv a\pmod{q}$ and $\{\alpha^{-1}m+\delta\}\in \mathcal{I}$. Then we have
\begin{equation}\label{SecondBeattyBound}
S_{\alpha,\beta,c;q,a}(N)=\sum_{\substack{m\leq M\\m\equiv a\pmod{q}}} \Lambda(m)\psi_\Delta(\alpha^{-1}m+\delta)+O(1+V(\mathcal{I},M))\,.
\end{equation}
We bound $V(\mathcal{I},M)$ using the bound on the discrepancy of (\ref{BeattyUniformDistribution}). We have
\begin{equation}\label{VIMBound}
V(\mathcal{I},M)\ll \Delta N+N^{1-\epsilon}\,,
\end{equation}
where the implied constant depends only on $\alpha$.  Now we evaluate the main term using the Fourier expansion for $\psi_\Delta$. We have 
\begin{equation}\label{psid}
\begin{split}
\sum_{\substack{m\leq M\\m\equiv a\pmod{q}}} \Lambda(m)\psi_\Delta(m)
= \gamma \sum_{\substack{m\leq M\\m\equiv a\pmod{q}}}\Lambda(m)
&+\sum_{k\geq 1} g_k e(k\delta) \sum_{\substack{m\leq M\\m\equiv a\pmod{q}}} \Lambda(m)e(\alpha^{-1}km)\\
&+\sum_{k\geq 1} h_k e(-k\delta) \sum_{\substack{m\leq M\\m\equiv a\pmod{q}}} \Lambda(m)e(-\alpha^{-1}km)\,.
\end{split}
\end{equation}
We show that every term in the right-hand side except for the first is small.  Using Lemma \ref{BSlemma}, we first determine a bound for terms with $k$ small. In particular, using that $\alpha^{-1}$ has finite type, for any $0\leq a<q\leq M^{\epsilon/4}$ we have the bound
\begin{equation}\label{gk1}
\sum_{k\leq M^\epsilon} g_k e(k\delta) \sum_{\substack{m\leq M\\m\equiv a\pmod{q}}} \Lambda(m)e(\alpha^{-1}km)\ll M^{1-\eta} \sum_{k\leq M^{\epsilon}} k^{-1}\ll M^{1-\eta/2}\,,
\end{equation}
with $\epsilon$ and $\eta$ the constants determined by Lemma \ref{BSlemma}. In addition, we have
\begin{equation}\label{hk1}
\sum_{k\leq M^\epsilon} h_k e(-k\delta) \sum_{\substack{m\leq M\\m\equiv a\pmod{q}}} \Lambda(m)e(-\alpha^{-1}km)\ll  M^{1-\eta/2}\,.
\end{equation}
For larger values of $k$, a trivial bound suffices. Indeed, using the Prime Number Theorem for arithmetic progressions and (\ref{gkhk}), we have
\begin{equation}\label{gk2}
\sum_{k> M^\epsilon} g_k e(k\delta) \sum_{\substack{m\leq M\\m\equiv a\pmod{q}}} \Lambda(m)e(\alpha^{-1}km)\ll M \sum_{k>M^\epsilon} g_k \ll M^{1-\epsilon} \Delta^{-1}\,
\end{equation}
and
\begin{equation}\label{hk2}
\sum_{k> M^\epsilon} h_k e(-k\delta) \sum_{\substack{m\leq M\\m\equiv a\pmod{q}}} \Lambda(m)e(-\alpha^{-1}km) \ll M^{1-\epsilon}\Delta^{-1}\,.
\end{equation}
Combining (\ref{psid}) --- (\ref{hk2}), we have the equation
$$
\sum_{\substack{m\leq M\\m\equiv a\pmod{q}}} \Lambda(m)\psi_\Delta(m)
= \gamma \, \sum_{\substack{m\leq M\\m\equiv a\pmod{q}}}\Lambda(m) +O(M^{1-\eta/2}+M^{1-\epsilon}\Delta^{-1})\,,
$$
where $\eta$ and $\epsilon$ depend only on $\alpha$, and the implied constant on $\alpha$, $\beta$, and $c$. Combining this equation with our original expression (\ref{SecondBeattyBound}) for $S_{\alpha,\beta,c;q,a}(N)$ and (\ref{VIMBound}), we have
$$
S_{\alpha,\beta,c;q,a}(N)=\gamma \sum_{\substack{m\leq M\\m\equiv a\pmod{q}}}\Lambda(m) +O(\Delta N+N^{1-\epsilon}+N^{1-\eta/2}+N^{1-\epsilon}\Delta^{-1})\,.
$$
Setting $\Delta=N^{-\epsilon/4}$ produces the desired result.
\end{proof} 
\begin{remark}
Theorem 5.4 of \cite{BS09} is precisely the case $c=1$ of this lemma, and is proved similarly.
\end{remark}
We now prove Part \ref{hyp2} of Hypothesis \ref{Hypothesis1}.
\begin{proposition}\label{BeattyHyp2} 
Let $\alpha$, $c$, $\mathcal{A}$, $\mathcal{P}$, and $\mathcal{L}$ be as in the statement of Theorem \ref{BeattyHyp}. Then there exists a constant $\theta>0$ depending only on $\alpha$ such that for any $L\in \mathcal{L}$ and any $A>0$ we have
$$
\sum_{q\leq x^\theta} \max_{(L(a),q)=1} \Bigl\lvert \#\mathcal{P}_{L,\mathcal{A}}(x;q,a)-\frac{\#\mathcal{P}_{L,\mathcal{A}}(x)}{\phi(q)}\Bigr\rvert \ll_A \frac{\#\mathcal{P}_{L,\mathcal{A}}(x)}{(\log x)^A}\,.
$$
\end{proposition}

\begin{proof}[Proof of Proposition \ref{BeattyHyp2}]
Using partial summation on Lemma \ref{PNTBeatty} there is some $\kappa>0$ such that for all $q\leq x^\kappa$ we have the estimates
$$
\#\mathcal{P}_{L_i,\mathcal{A}}(x;q,a)=c \alpha^{-1} \pi(x;q,a)+O(N^{1-\kappa})\,,
$$
and
$$
\#\mathcal{P}_{L_i,\mathcal{A}}(x)=c \alpha^{-1} \pi(x)+O(N^{1-\kappa})\,.
$$
Then, taking a sum over all $q\leq x^{\kappa/2}$, we obtain the bounds
\begin{equation}\label{BVBxqa}
\sum_{q\leq x^{\kappa/2}} \max_{(a,q)=1}\left\lvert \#\mathcal{P}_{L_i,\mathcal{A};q,a}(x)-c\alpha^{-1} \pi(x;q,a)\right\rvert \ll N^{1-\kappa/2}\,,
\end{equation}
and
\begin{equation}\label{BVBx}
\sum_{q\leq x^{\kappa/2}} \left\lvert\frac{ \#\mathcal{P}_{L_i,\mathcal{A}}(x)}{\phi(q)}-\frac{c}{\alpha \phi(q)} \pi(x)\right\rvert \ll N^{1-\kappa/2}\,.
\end{equation}
Combining (\ref{BVBxqa}) and (\ref{BVBx}) with the Bombieri-Vinogradov theorem and using the triangle inequality hence produces
\begin{equation}\label{BVBeatty}
\sum_{q\leq x^{\kappa/2}} \max_{(a,q)=1}\left\lvert \#\mathcal{P}_{L_i,\mathcal{A};q,a}(x)-\frac{\#\mathcal{P}_{L_i,\mathcal{A}}(x)}{\phi(q)}\right\rvert \ll_A\frac{N}{(\log N)^A}\,.
\end{equation}
Proposition \ref{BeattyHyp2} follows with $\theta=\kappa/2$ by evaluating the bound of (\ref{BVBeatty}) at $2x$ and $x$ and subtracting.
\end{proof}

We now establish Theorem \ref{BeattyHyp} by proving the other two parts of Hypothesis \ref{Hypothesis1} for some positive $\theta$ and our choice of $\mathcal{A}$, $\mathcal{P}$, and $\mathcal{L}$.
\begin{proof}[Proof of Theorem \ref{BeattyHyp}]
Let $\alpha$ be a irrational number of finite type $\tau$ and let $\mathcal{A}$ and $\mathcal{P}$ be as in the statement of the lemma. By (\ref{BeattyUniformDistribution}), the discrepancy $D_N(\alpha)$ of the sequence $\alpha, 2\alpha,\ldots N\alpha$ satisfies $D_N(\alpha)=O(N^{(-1/\tau)+\epsilon})$ for any positive $\epsilon$. We will use this to prove Parts \ref{hyp1} and \ref{hyp3} of Hypothesis \ref{Hypothesis1}. 

Fix $\epsilon>0$ small, and let $C>0$ be a constant such that the inequality $D_N(\alpha)<C N^{(-1/\tau)+\epsilon}$ holds. We show that $D_N(\alpha/q)\ll  (N/q)^{(-1/\tau)+\epsilon}$ for any $q\in \mathbb{Z}^+$, where $q$ may vary with $N$. For, using the triangle inequality 
to divide the expression for the discrepancy into $q$ pieces, we have
$$
D_N\left(\frac{\alpha}{q}\right) \leq \frac{1}{N}\sum_{0< a\leq q} \left\lceil \frac{N-a+1}{q}\right\rceil D_{\left\lceil \frac{N-a+1}{q}\right\rceil}\left(\frac{a \alpha}{q}, \alpha+\frac{a\alpha}{q},\ldots, \left\lfloor\frac{N-a}{q}\right\rfloor\alpha+\frac{a\alpha}{q}\right)\,.
$$
Since a translation of a sequence sends a interval of values modulo 1 to at most 2 intervals, we have $D_{\left\lceil \frac{N-a+1}{q}\right\rceil}\left(\frac{a\alpha}{q},\ldots, \left\lfloor\frac{N-a}{q}\right\rfloor\alpha+\frac{a\alpha}{q}\right) \leq 2 D_{\left\lceil \frac{N-a+1}{q}\right\rceil}\left(a\alpha,\ldots,\left\lceil\frac{N-a+1}{q}\right\rceil\alpha\right)$. Then, using our above bound on $D_N(\alpha)$, we have
\begin{equation}
D_N(\alpha/q)\leq \frac{2}{N}\sum_{0<a\leq q} \left\lceil \frac{N-a+1}{q}\right\rceil C \left(\left\lceil \frac{N-a+1}{q}\right\rceil\right)^{(-1/\tau)+\epsilon} \leq \frac{2Cq}{N} \left(\frac{N+q}{q}\right)^{(-1/\tau)+\epsilon+1}.
\end{equation}
Assuming $q<N$ we then have 
\begin{equation}\label{BeattyqDiscrepancy}
D_N\left(\frac{\alpha}{q}\right)\ll\left(\frac{N}{q}\right)^{(-1/\tau)+\epsilon}
\end{equation}
with an implied constant independent of $q$.

Define $G_{\mathcal{A}}(x;q,a):=\#(\mathcal{A}\cap \{a,a+q,a+2q,\ldots,a+\lfloor \frac{x-a}{q}\rfloor q\}) $ and define $G_{\mathcal{A}}(x):=\#(\mathcal{A}\cap\{1,\ldots,x\})$. We then have that $G_{\mathcal{A}}(x;q,a)$ is the number of $n$ such that $\frac{a}{q}\leq \left\{\frac{\alpha n}{q}\right\} < \frac{a+c}{q}$ and $\alpha n\leq x$, via the correspondence $n\mapsto \lfloor \alpha n\rfloor$. Then, by (\ref{BeattyqDiscrepancy}), we have
$$
G_{\mathcal{A}}(x;q,a)=\frac{cx}{\alpha q}+O\left(x\left(\frac{x}{q}\right)^{(-1/\tau)+\epsilon}\right)\,,
$$
and setting $q=1$ gives us the bound
$$
G_{\mathcal{A}}(x)=\frac{cx}{\alpha}+O\left(x^{(-1/\tau)+\epsilon+1}\right)\,.
$$
Suppose $q<N^{\frac{1}{2\tau}-\epsilon}$. Then the two preceding equations gives us
\begin{equation}\label{GAqbound}
\left\lvert G_{\mathcal{A}}(x;q,a)-\frac{G_{\mathcal{A}}(x)}{q}\right\rvert \ll x^{1-\frac{1}{2\tau}}\,.
\end{equation}
Evaluating (\ref{GAqbound}) at $x$ and $2x$, then subtracting, produces the bound
\begin{equation}
\left\lvert \#\mathcal{A}(x;q,a)-\frac{\#\mathcal{A}(x)}{q}\right\rvert \ll x^{1-\frac{1}{2\tau}}\,.
\end{equation}
Parts \ref{hyp1} and \ref{hyp3} of Hypothesis \ref{Hypothesis1} follow immediately from this bound when $\theta=\frac{1}{2\tau}-\epsilon$. Since $\epsilon>0$ was arbitrary, parts \ref{hyp1} and \ref{hyp3} hold for all $\theta<\frac{1}{2\tau}$. 

As such, since we have already proven Part \ref{hyp2} of Hypothesis \ref{Hypothesis1} as Proposition \ref{BeattyHyp2}, with $\theta=\frac{\kappa}{2}$, we have that $(\mathcal{A},\mathcal{P},\mathcal{L}, \theta)$ satisfies Hypothesis \ref{Hypothesis1} for any $\mathcal{L}$ whenever $\theta <\min(\frac{\kappa}{2},\frac{1}{2\tau})$, as was to be shown.
\end{proof}

\section{Proof of Theorem \ref{LeitmannBound}} \label{sectleitmann}
To prove Theorem \ref{LeitmannBound}, we show that Leitmann functions satisfy Hypothesis \ref{Hypothesis1}. For the rest of the section, fix a Leitmann function $f:[c_0,\infty)\rightarrow \mathbb{R}^+$ with $f(c_0)=c_3$, and let $\alpha_1$, $\alpha_2$, $\alpha_3$ be as in Definition \ref{defleitmanntype}. Let 
\begin{equation*} \mathcal{A}=\{\lfloor f(n)\rfloor:n\in \mathbb{Z}\cap [c_0,\infty)\}\,.\end{equation*} We show that $\mathcal{A}$ and $\mathbb{P}$ satisfy Hypothesis \ref{Hypothesis1} for some positive $\theta$ and any set of linear forms and any finite admissible set of linear forms $\mathcal{L}=\{L_1,\ldots,L_k\}$ with $L_i(n)=n+l_i$.  Let $g$ denote the inverse function of $f$.

\begin{lemma}\label{LeitmannHyp2}
The set $\mathcal{A},\mathcal{P},\mathcal{L}$ described above satisfies part \ref{hyp2} of Hypothesis \ref{Hypothesis1} for some positive constant $\theta$ that does not depend on $\mathcal{L}$.
\end{lemma}
\begin{proof}
Consider the translated function $f_{L_i}(n):=f(n)+l_i$. By Part \ref{Leitmann-Vinogradov} of Theorem \ref{leitmannthm} applied to $f_{L_i}$ at $N=x$ and $N=2x$, for some $\theta>0$  we have
\begin{equation}\label{411}
\sum_{q\leq x^\theta} \max_{(l,q)=1}\left| \#\mathcal{P}_{L_i,\mathcal{A}}(x;q,l)-\frac{1}{\phi(q)} \int_{x}^{2x} \frac{g'(t)}{\log t} dt\right| \ll_A \frac{g(2x)+g(x)}{(\log N)^A}\,,
\end{equation}

By Part \ref{thmleitmannbound1} of Theorem \ref{leitmannthm} we also have the trivial bound
\begin{equation}\label{412}
\sum_{q\leq N^{\theta}} \frac{1}{\phi(q)}  \left\vert \#\mathcal{P}_{L_i,\mathcal{A}}(x) - \int_{x}^{2x} \frac{g'(t)}{\log t} dt\right| \ll_A  \frac{g(x)+g(2x)}{(\log N)^A}\,.
\end{equation}
Combining  (\ref{411}) and (\ref{412}) using the triangle inequality, we have
\begin{equation}\label{BVinter}
\sum_{q\leq N^{\theta}} \max_{(l,q)=1}\left| \#\mathcal{P}_{f_{L_i}}(y;q,l)-\frac{1}{\phi(q)}\#\mathcal{P}_{f_{L_i}}(y;1,0)\right| \ll_A \frac{g(N)}{(\log N)^A}\,,
\end{equation}
completing the proof.

\end{proof}

Now we must establish parts \ref{hyp1} and \ref{hyp3} of Hypothesis \ref{Hypothesis1}. As in the Beatty prime case above, it will follow from the following bound on the discrepancy of the sequence $\{\frac{f(n)}{q}\}_{n\in\mathbb Z}$.
\begin{proposition}\label{LeitmannDisc}
There is a positive absolute constant $\theta$ such that discrepancy $D_N\left(\frac{f(n)}{q}\right)$ of the sequence $\frac{f(c_0)}{q},\ldots,\frac{f(N+c_0-1)}{q}$ satisfies the bound
$$
D_N\left(\frac{f(n)}{q}\right)\ll N^{-\theta}\,,
$$
when $q\ll N^{\theta}$. The implied constant depends only on $f$ and $\theta$.
\end{proposition}
\begin{proof}
We will show the result when $\theta=1/11$. By Theorem  \ref{ETT},
 we have, for any positive integer $m$, the bound
$$
D_N\left(\frac{f(n)}{q}\right)\ll\frac{1}{m}+\sum_{h=1}^m \frac{1}{hN}\left\lvert \sum_{n=1}^{N}e^{2\pi ih f(n+c_0-1)/q}\right\rvert\,.
$$
We have two cases, depending on whether $\alpha_2=0$ or $\alpha_2>0$. 

Suppose first that $\alpha_2$ is positive. Let $\rho_N=\inf_{x\in[c_0,N+c_0]} f''(n)$. By Theorem \ref{2.7} we have the bound 
\begin{equation}\label{ET}
\left\lvert \sum_{n=1}^{N}e^{2\pi ih f(n+c_0-1)/q}\right\rvert \leq \left(\frac{h}{q}\left\lvert f'(N+c_0-1)-f'(c_0)\right\rvert+2\right)\left(\frac{4}{\sqrt{h\rho_N/q}}+3\right)\,.
\end{equation}
Now, we have $f''(x)\sim \frac{\alpha_1 \alpha_2 f(x)}{x^2}=o(1)$ and $f''(x)>0$, so we also have $\rho_N\sim \frac{\alpha_1\alpha_2 f(N)}{N^2}$. In addition, we have $f'(x)\sim \alpha_1 f(x)/x$. It follows that
\begin{align*}
\left\lvert \sum_{n=1}^{N}e^{2\pi ih f(n+c_0-1)/q}\right\rvert &\ll \left(\frac{hf(N)}{qN} +1\right)\left(\frac{N\sqrt{q}}{\sqrt{h f(N)}}+1\right)\\
&\ll \frac{hf(N)}{qN}+\frac{N\sqrt{q}}{\sqrt{hf(N)}}+\sqrt{\frac{hf(N)}{q}}+1\,,
\end{align*}
where the implied constant only depends on $f$. We insert this expression into (\ref{ET}) with $m=N^\kappa$. We have
$$
D_N\left(\frac{f(n)}{q}\right)\ll N^{-\kappa}+\frac{N^\kappa f(N)}{qN^2}+\frac{\sqrt{q}}{f(N)N^{\kappa/2}}+\frac{N^{\kappa/2}\sqrt{f(N)}}{N\sqrt{q}}+\frac{\log N}{N}\,.
$$
Since $N\ll f(N) \ll N^{12/11-\epsilon}$ for some positive $\epsilon$, we set $\kappa=1/11$, at which point we deduce the following:
$$
D_N\left(\frac{f(n)}{q}\right)\ll N^{-1/11}+\frac{\sqrt{q}}{N^{23/22}}\,.
$$
The result follows with $\theta=1/11$.

Now suppose $\alpha_2=0$. Then we have $x f''(x)=f'(x) s(x) t(x)$, where $s(x)=o(1)$ is a decreasing function, $s(x)\gg x^{-\epsilon}$ for all positive $\epsilon$  and $1\ll t(x)\ll 1$. Then, in particular, we have $\rho_N\gg \frac{s(N)f(N)}{N^2}$, so (\ref{ET}) simplifies to the following bound.
$$
\left\lvert \sum_{n=1}^{N}e^{2\pi ih f(n+c_0-1)/q}\right\rvert \ll \frac{h f(N)}{qN}+\frac{N\sqrt{q}}{\sqrt{h s(n)f(N)}}+\sqrt{\frac{hf(N)}{qs(N)}}+1\,.
$$
Since $1/s(N)=o(N^\epsilon)$ for all positive $\epsilon$ the same argument as above produces the same result, that
$$
D_N\left(\frac{f(n)}{q}\right)\ll N^{-1/11}+\frac{\sqrt{q}}{N^{23/22}}\,.
$$
The result follows.
\end{proof}

\begin{corollary}\label{LeitmannHyp13}
Parts \ref{hyp1} and \ref{hyp3} of Hypothesis \ref{Hypothesis1} hold for our choice of $\mathcal{A}, \mathcal{P}$, some positive choice of $\theta$, and any set of linear forms $\mathcal{L}=\{L_1,\ldots,L_k\}$ with $L_i(n)=n+l_i$.
\end{corollary}
\begin{proof}
Following the same argument as in the proof of Theorem \ref{BeattyHyp}, Proposition \ref{LeitmannDisc} implies Parts \ref{hyp1} and \ref{hyp3} of Hypothesis \ref{Hypothesis1} whenever $\theta$ is less than half the $\theta$ of Proposition \ref{LeitmannDisc}; in particular, this result follows for all $\theta<1/22$.
\end{proof}

Between Lemma \ref{LeitmannHyp2} and Corollary \ref{LeitmannHyp13} we have that $\mathcal{A}$ and $\mathcal{P}$  satisfy Hypothesis \ref{Hypothesis1} for some positive choice of $\theta$ and any admissible $\mathcal{L}$. It remains to use this result to prove Theorem \ref{LeitmannBound}. 
\begin{proof}[Proof of Theorem \ref{LeitmannBound}]
Let $f:[c_0,\infty)\rightarrow \mathbb{R}^+$ be a Leitmann function. 

We first find $\delta$ such that 
\begin{equation}\label{3.1DeltaHyp}
\frac{1}{k} \sum_{L\in\mathcal{L}} \#\mathcal{P}_{L,\mathcal{A}}(x) \geq \delta \frac{\#\mathcal{A}(x)}{\log x}
\end{equation}
 for all $\mathcal{L}$ for $x$ sufficiently large. But by part \ref{thmleitmannbound1} of Theorem \ref{leitmannthm}, we have  $\#\mathcal{P}_{L, \mathcal{A}}(x) \sim \frac{\#\mathcal{A}(x)}{\log x}$. 
As such, for any $0<\delta<1$, (\ref{3.1DeltaHyp}) holds for all sufficiently large $x$, so we can take $\delta=1-\epsilon$ with $\epsilon>0$ in (\ref{3.1DeltaHyp}). Then, by Theorem \ref{thmmaynard}, we have
\begin{equation}
\#\{n \in \mathcal{A}(x) : \#(\{L_1(n), \ldots, L_k(n)\} \cap \mathcal{P}) \ge C^{-1} \delta \log k \} \gg \frac{\#\mathcal{A}(x)}{(\log x)^k \exp(Ck)}\,,
\end{equation}
where $C$ only depends on $\theta$. 
\end{proof}
\section*{Acknowledgments}
We would like to thank Jesse Thorner for his patient guidance throughout the process, and Ken Ono for his advice and encouragement. This paper was written while all the authors were participants in the 2014 Emory Math REU, and as such we would like to thank the NSF for its support and the Emory Department of Mathematics and Computer Science for its hospitality.
\bibliographystyle{plain}
\bibliography{CItes}

\begin{thebibliography}{10}

\bibitem{BS09}
William~D. Banks and Igor~E. Shparlinski.
\newblock Prime numbers with {B}eatty sequences.
\newblock {\em Colloq. Math}, 115(2):147--157, 2009.

\bibitem{Dav}
Harold Davenport and Hugh~L. Montgomery.
\newblock {\em Multiplicative Number Theory}.
\newblock Springer, 3rd edition, 2000.

\bibitem{GPY09}
Daniel Goldston, J{\'a}nos Pintz, and Cem Y{\i}ld{\i}r{\i}m.
\newblock Primes in tuples {I}.
\newblock {\em Ann. of Math.}, 170(2):819--862, 2009.

\bibitem{KN12}
Lauwerens Kuipers and Harald Niederreiter.
\newblock {\em Uniform distribution of sequences}.
\newblock Courier Dover Publications, 2012.

\bibitem{Lei77}
Dieter Leitmann.
\newblock The distribution of prime numbers in sequences of the form $[f(n)]$.
\newblock {\em Proc. London Math. Soc.}, 35(3):448--462, 1977.

\bibitem{May14}
James Maynard.
\newblock Dense clusters of primes in subsets.
\newblock {\em Preprint}, 2014.

\bibitem{May13}
James Maynard.
\newblock Small gaps between primes.
\newblock {\em Ann. of Math.}, To appear.

\bibitem{P8B}
D.H.J. Polymath.
\newblock Variants of the selberg sieve, and bounded intervals containing many
  primes.
\newblock {\em Preprint}, 2014.

\bibitem{Tho14}
Jesse Thorner.
\newblock Bounded gaps between primes in chebotarev sets.
\newblock {\em Res. Math. Sci.}, 1(4), 2014.

\bibitem{Vin04}
Ivan Vinogradov.
\newblock {\em The method of trigonometrical sums in the theory of numbers}.
\newblock Dover Publications, 2004.

\bibitem{Zha14}
Yitang Zhang.
\newblock Bounded gaps between primes.
\newblock {\em Ann. of Math.}, 179(3):1121--1174, 2014.

\end{thebibliography}

\end{document}